\documentclass[12pt]{amsart}
\usepackage{amsmath,amsthm,amssymb,amsfonts,enumerate,color}
\usepackage{amsmath, amsthm, amsfonts, amssymb, color}
\usepackage{mathrsfs}
\usepackage{amsmath, amsthm, amsfonts, amssymb, color}
\usepackage{mathrsfs}
\usepackage{amsfonts, amsmath}
\usepackage{amsmath,amstext,amsthm,amssymb,amsxtra}
\usepackage{txfonts} 
\usepackage[colorlinks, citecolor=blue,pagebackref,hypertexnames=false]{hyperref}
\allowdisplaybreaks
 \usepackage{pgf,tikz}
\usepackage{stmaryrd}
\begin{document}




\newcommand{\norm}[1]{\left\Vert#1\right\Vert}
\newcommand{\abs}[1]{\left\vert#1\right\vert}
\newcommand{\set}[1]{\left\{#1\right\}}
\newcommand{\Real}{\mathbb{R}}
\newcommand{\RR}{\mathbb{R}^n}
\newcommand{\supp}{\operatorname{supp}}
\newcommand{\card}{\operatorname{card}}
\renewcommand{\L}{\mathcal{L}}
\renewcommand{\P}{\mathcal{P}}
\newcommand{\T}{\mathcal{T}}
\newcommand{\A}{\mathbb{A}}
\newcommand{\K}{\mathcal{K}}
\renewcommand{\S}{\mathcal{S}}
\newcommand{\blue}[1]{\textcolor{blue}{#1}}
\newcommand{\red}[1]{\textcolor{red}{#1}}
\newcommand{\Id}{\operatorname{I}}

\newtheorem{thm}{Theorem}[section]
\newtheorem{prop}[thm]{Proposition}
\newtheorem{cor}[thm]{Corollary}
\newtheorem{lem}[thm]{Lemma}
\newtheorem{lemma}[thm]{Lemma}
\newtheorem{exams}[thm]{Examples}
\theoremstyle{definition}
\newtheorem{defn}[thm]{Definition}
\newtheorem{rem}[thm]{Remark}

\numberwithin{equation}{section}

\title[Neumann boundary condition]
{ Application BMO type space to parabolic equations of Navier-Stokes type with the Neumann boundary condition}

 \author[M.H. Yang and C. Zhang]{Minghua Yang\  and\ Chao Zhang}

 \address{Department of Mathematics\\Jiangxi University of Finance and Economics \\
         Nanchang 330032, PR China}
 \email{ymh20062007@163.com}

 \address{School of Statistics and Mathematics \\
             Zhejiang Gongshang University \\
             Hangzhou 310018, PR China}
 \email{zaoyangzhangchao@163.com}

\thanks{M. H. Yang was supported by the Postdoctoral Science Foundation of Jiangxi Province (grant no. 2017KY23) and Educational Commission Science Programm of Jiangxi Province (grant no. GJJ170345). C. Zhang was  supported by the Natural Science Foundation of Zhejiang Province (Grant No. LY18A010006) and the National Natural Science Foundation of China (Grant No. 11401525)}
 \subjclass[2010]{42B35, 42B37, 35J10,  47F05}
\keywords{Neumann operator,  BMO space, Carleson measure, Navier-Stokes equation}

\begin{abstract}  Let $\L$ be a Neumann operator of the form $\L=-\Delta_{N}$ acting on $L^2(\mathbb R^n)$.
 Let ${\rm BMO}_{\Delta_{N}}(\RR)$ denote the BMO
space on $\RR$ associated to the Neumann operator $\L$.
In this article we will
show that a function  $f\in {\rm BMO}_{\Delta_{N}}(\RR)$ is the trace of the solution of
${\mathbb L}u=u_{t}+\L u=0,  u(x,0)= f(x),$
 where $u$ satisfies  a Carleson-type condition
\begin{eqnarray*}
 \sup_{x_B, r_B} r_B^{-n}\int_0^{r_B^2}\int_{B(x_B, r_B)}  |\nabla   u(x,t)|^2 {dx dt }  \leq C <\infty,
\end{eqnarray*}
for some constant $C>0$.
Conversely, this Carleson   condition characterizes  all the ${\mathbb L}$-carolic functions whose traces belong to
the space ${\rm BMO}_{\Delta_{N}}(\RR)$.
This result extends the  analogous characterization founded by E. Fabes and U. Neri in  (\textit{Duke Math. J.}
\textbf{42} (1975),  725-734) for  the classical BMO space  of John and Nirenberg. Furthermore, based on the characterization of ${\rm BMO}_{\Delta_{N}}(\RR)$ space mentioned above, we prove global well-posedness for
parabolic equations of Navier-Stokes type with the Neumann boundary condition under smallness condition on intial data $u_{0}\in {{\rm BMO}_{\Delta_{N}}^{-1}(\RR)}$, which is motivated by the work of
P. Auscher and D. Frey  (\textit{J. Inst. Math. Jussieu}
\textbf{16(5)} (2017),  947-985).
\end{abstract}

\maketitle


\section{Introduction and statement of the main result}
\setcounter{equation}{0}

In Harmonic Analysis, to study a (suitable) function $f(x)$ on $\mathbb R^n$ is to consider a harmonic function
on $\mathbb R^{n+1}_{+}$ which has the boundary value as $f(x)$. A standard choice for such a harmonic
function is the Poisson
integral $e^{-t\sqrt{-\Delta}} f(x)$ and one recovers $f(x)$ when letting $t \rightarrow 0^{+}$, where $\Delta=\sum_{i=1}^n\partial_{x_i}^2$ is the Laplace operator.
 In other words,
one obtains $u(x,t) = e^{-t\sqrt{-\Delta}} f(x)$ as the solution of the equation
\begin{equation*}
\left\{
\begin{aligned}
\partial_{tt}u +\Delta u=0&,\ \ \ \  \ \ \ \  x\in \RR, \ t>0, \\
u(x,0)= f(x)&, \ \ \ \ \ \ \ \ x\in \RR.\\
\end{aligned}
\right.
\end{equation*}
This approach is intimately related to
the study of singular integrals. In \cite{SW}, the authors studied the classical case $f \in L^p(\mathbb R^n)$, $1 \le p \le \infty$.

It is well known that the BMO space, i.e. the space
of functions of bounded mean oscillation, is natural substitution to study singular integral at the end-point space $L^{ \infty}(\RR)$.
A celebrated   theorem of  Fefferman and Stein \cite{FS}   states    that
 a BMO function   is the trace of the solution of
 $\partial_{tt}u +\Delta u=0,  u(x,0)= f(x),$
 whenever $u$ satisfies
\begin{eqnarray}\label{ee1.1}
 \sup_{x_B, r_B} r_B^{-n}\int_0^{r_B}\int_{B(x_B, r_B)}|t\nabla  u(x,t)|^2 {dx dt\over t }<\infty,
\end{eqnarray}
where $\nabla=(\nabla_x, \partial_t)=(\partial_{1},...,\partial_{n}, \partial_t).$  Conversely,
   Fabes, Johnson and Neri \cite{FJN} showed that  condition  above
   characterizes  all the harmonic functions whose traces are in ${\rm BMO}(\RR)$ in 1976.
  The study of this topic has been widely
extended to more general operators such as elliptic operators and Schr\"odinger operators (instead of the Laplacian),  for more general initial data spaces such as Morrey spaces
and for domains other than $\mathbb R^n$ such as Lipschitz domains. For these generalizations,   see \cite{DKP, DYZ, FN1, FN,Song}.

In \cite{FN1}, Fabes and Neri further generalized  the above characterization to caloric functions (temperature), that is the authors proved that a BMO function $f$ is the trace of the solution of

\begin{equation*}
\left\{
\begin{aligned}
\partial_{t}u-\Delta u=0&,\ \ \ \  \ \ \ \  x\in \RR, \ t>0, \\
u(x,0)= f(x)&, \ \ \ \ \ \ \ \ x\in \RR,\\
\end{aligned}
\right.
\end{equation*}
 whenever $u$ satisfies
\begin{eqnarray} \label{e1.1}
 \sup_{x_B, r_B} r_B^{-n}\int_0^{r_B^2}\int_{B(x_B, r_B)}  |\nabla  u(x,t)|^2 {dx dt }  <\infty,
\end{eqnarray}
where $\nabla=(\nabla_x, \partial_t);$  and, conversely,  the  condition \eqref{e1.1}
   characterizes  all the carolic functions whose traces are in ${\rm BMO}(\RR)$.
The authors in \cite{JX} made a complete conclusion, related to harmonic functions and carolic functions, about this subject.

 We denote by $\Delta_n$ the Laplacian on $\mathbb{R}^{n}$. Next we recall the Neumann Laplacian on $\mathbb{R}^{n}_{+}$
and $\mathbb{R}^{n}_{-}$.
 Consider the Neumann problem on the half line $(0, \infty)$:
\begin{eqnarray}\label{AAA1.1}
\left\{
\begin{array}{ll}
 u_t-\Delta u=0,  \ \ \
 \ \ \   &{\rm }\ 0<x<\infty,\, 0<t<\infty, \\[6pt]
 u(x,0)=f(x),  \ \ \
 \ \ \   &{\rm }\ 0<x<\infty, \\[6pt]
 u_{x}(0,t)=0,  \ \ \
 \ \ \   &{\rm }\ 0<t<\infty. \\[6pt]
\end{array}
\right.
\end{eqnarray}
Denote this corresponding Laplacian by $\Delta_{1, N_{+}}$ and we
see that $$u(x, t) =e^{t\Delta_{1,N_{+}}}(f)(x).$$
For $n > 1$, we
write $\mathbb{R}^{n}_{+}= \mathbb{R}^{n-1}\times \mathbb{R}_{+}$. And we definition the Neumann
Laplacian on $\mathbb{R}^{n}_{+}$ by
$\Delta_{n, N_{+}}=\Delta_{n-1}+\Delta_{1, N_{+}}$;
where $\Delta_{n-1}$ is the Laplacian on $\mathbb{R}^{n-1}$ and $\Delta_{1, N_{+}}$ is the
Laplacian corresponding to \eqref{AAA1.1}, for more results related to this topic, we refer the readers to see \cite{DDSY, LW}. Similarly we can definition Neumann Laplacian $\Delta_{n, N_{-}}$ on $\mathbb{R}^{n}$. Now let $\Delta_{N}$ be the uniquely determined unbounded operator acting on $L^{2}(\mathbb{R}^{n})$ such
that
$({\Delta_{N}}f)_{+}={ \Delta_{N_{+}}}f_{+}$ and $({\Delta_{N}}f)_{-}={\Delta_{N_{-}}}f_{-}$ for all $f: \mathbb{R}^{n}\rightarrow\mathbb{R}$ such that $f_{+}\in W^{1,2}(\mathbb{R}_{+}^{n})$ and $f_{-}\in W^{1,2}(\mathbb{R}_{-}^{n})$. Then $\Delta_{N}$ is a positive self-adjoint operator and $$(\exp(t\Delta_{N})f)_{+}=\exp(t\Delta_{N_{+}})f_{+};\ \ \ \  \ (\exp(t\Delta_{N})f)_{-}=\exp(t\Delta_{N_{-}})f_{-}.$$
The operator   $\Delta_{N}$ is a self-adjoint
operator on $L^2(\RR)$. Hence $\Delta_{N}$ generates the $\Delta_{N}$-heat semigroup $$\T_tf(x)=e^{t\Delta_{N}}f(x)=\int_{\Real^n}p_{t,\Delta_{N}}(x,y)f(y)dy,\ f\in L^2(\RR),\ t>0.$$
 The heat kernel of $\rm \exp(t\Delta_{N})$, denoted by $p_{t, \Delta_{N}}(x, y)$, is
 then given as:
 $$p_{t, \Delta_{N}}(x, y)=\frac{1}{(4\pi t)^{\frac{n}{2}}}e^{-\frac{|x'-y'|^{2}}{4t}}(
 e^{-\frac{|x_{n}-y_{n}|^{2}}{4t}}-e^{-\frac{|x_{n}+y_{n}|^{2}}
 {4t}})H(x_{n}y_{n}),$$
 where $H: R\rightarrow\{0,1\}$ is the Heaviside given by
 $$ H(t)=0, \,  t<0; \ \ \ \ \, H(t)=1, \, t\geq0,$$
 $x'=(x_{1}, x_{1}, ... , x_{n-1})$ and $y'=(y_{1},y_{1}, ... , y_{n-1})$.
It is well-known that  the semigroup kernels $p_{t, \Delta_{N}}(x, y)$  of the operators $e^{t\Delta_{N}}$ satisfies Gaussian bounds:
\begin{eqnarray*}
0\leq p_{t, \Delta_{N}}(x, y)\leq h_t(x-y)
\end{eqnarray*}
for all $x,y\in\RR$ and $t>0$, where
$$
h_t(x)=(4\pi t)^{-\frac{n}{2}}e^{-\frac{|x|^2}{4t}}
$$
is the kernel of the classical heat semigroup
$\set{{T}_t}_{t>0}=\{e^{t\Delta}\}_{t>0}$ on $\Real^n$. For the classical heat semigroup associated with Laplacian, see \cite{St1970}.
For $f\in  L^p(\RR)$, $1\leq p<  \infty,$
it is well known that $u(x,t)=e^{t\Delta_{N}}f(x), t>0, x\in\RR$, is a solution to the equation
\begin{eqnarray}\label{el.4}
{\mathbb L}u=\partial_{t}u-\Delta_{N}u =0\ \ \ {\rm in }\ {\mathbb R}^{n+1}_+
\end{eqnarray}
with the boundary data $f\in  L^p(\RR)$, $1\leq p<  \infty.$
 The equation  ${\mathbb L}u =0$ is interpreted in the weak sense via a sesquilinear form, that is, \ \   $u\in {W}^{1, 2}_{{\rm loc}} ( {\mathbb R}^{n}_\aleph \times\mathbb{R}^{+}) $ is a weak solution of ${\mathbb L}u =0$   if it satisfies
$$\int_{{\mathbb R}^{n}_\aleph \times\mathbb{R}^{+}}
{\nabla_x}u(x,t)\cdot {\nabla_x}\psi(x,t)\,dxdt-\int_{{\mathbb R}^{n}_\aleph \times\mathbb{R}^{+}} u(x,t) \partial_t\psi(x,t)dxdt=0,\ \ \ \
 $$
for any $ \psi(x,t)\in C_0^{1}( {\mathbb R}^{n}_\aleph \times\mathbb{R}^{+})$ and $\aleph\in \{+,\ \ -\}$. In the sequel,   we call such a function $u$  an ${\mathbb L}$-carolic function associated to the operator ${\mathbb L}$.

The first main aim of this article is to study a similar characterization to \eqref{e1.1}  for the Neumann operator $\Delta_{N}$.
In a word,  we  are interested in  deriving the characterization of the solution
 to the equation $ {\mathbb L}u  =0
 $ in ${\mathbb R}^{n+1}_+$ having boundary values with BMO type data ($\rm BMO_{\Delta_{N}}({\mathbb R}^n)$). We now recall the definition and some fundamental properties of $\rm BMO_{\Delta_{N}}({\mathbb R}^n)$ from \cite{ DDSY, LW}.

 Define
 $$M=\left\{f \in L_{loc}^{1}(\mathbb{R}^{n}): \exists\, d>0,\, s.t. \int_{\mathbb{R}^{n}}\frac{|f(x)|^{2}}{1+|x|^{n+d}}dx<\infty\right\}.$$

 \begin{defn}[\cite{DDSY}, Definition 2.2]
We say that $f\in M$ is of bounded mean oscillation
associated with $\Delta_{N}$, abbreviated as $\rm BMO_{\Delta_{N}}(\mathbb{R}^{n})$, if
\begin{equation}\label{A1.8}
\left\|f\right\|_{\rm BMO_{\Delta_{N}}(\mathbb{R}^{n})}=\sup_{B(y,r)}\frac{1}{\left|B(y,r)\right|} \int_{B(y,r)}\left|f(x)-\exp(-r^{2}\Delta_{N})f(x)\right|dx<\infty,
\end{equation}
where the supremum is taken over all balls $B(y, r)$ in $\mathbb{R}^{n}$. The smallest bound for which
\eqref{A1.8} is satisfied is then taken to be the norm of $f$ in this space, and is denoted by $\|f\|_{\rm BMO_{\Delta_{N}}(\mathbb{R}^{n})}$.
\end{defn}

Let us introduce a new  function class on  half plane $\mathbb{R}_+^{n+1}$ or $\mathbb{R}_-^{n+1}$.

\begin{defn}[Temperature Mean Oscillation for $\Delta_{\rm N}$]
A $C^1$-functions $u(x,t)$ defined on $\Real_+^{n}\times(0,\infty)$ belongs to the class  ${\rm TMO_{\Delta_{N_+}}}(\Real_+^{n}\times (0,\infty))$, if  $u(x,t)$ is
the solution of
\begin{eqnarray}\label{B1.1}
\left\{
\begin{array}{ll}
 \partial_tu-\Delta u=0,  \ \ \
 \ \ \   &{\rm }\ x\in \mathbb{R}_+^{n},\, 0<t<\infty, \\
 \partial_{x_n}u(x',0,t)=0, &{\rm }\ x'\in \mathbb{R}^{n-1},\, 0<x_n,\ t<\infty,\\
\end{array}
\right.
\end{eqnarray}
 such that
\begin{eqnarray}\label{e1.8}
\|u\|^2_{{\rm TMO_{\Delta_{N_+}}}(\Real_+^{n}\times(0,\infty))}&=& \sup_{B(x_B, r_B)\subset \mathbb{R}_+^n\times (0,\infty)}     r_B^{-n} \int_0^{r_B^2}\int_{B(x_B, r_B)}  | \nabla u(x,t) |^2
{dx dt }  <\infty,
\end{eqnarray}
where $\nabla=(\nabla_x, \partial_t).$
Similarly, we can define ${\rm TMO_{\Delta_{N_-}}}(\Real_-^{n}\times (0,\infty))$, the temperature mean oscillation for $\Delta_{\rm N_-}$.

{We say that $u(x,t) \in {\rm TMO_{\Delta_{N}}}(\Real^{n}\times (0,\infty))$, if   $u_+(x,t)=u(x,t)\chi_{\mathbb R_+^n\times (0,\infty)}(x,t)\in {\rm TMO_{\Delta_{N_+}}}(\Real_+^{n}\times (0,\infty))$
 and $u_-(x,t)=u(x,t)\chi_{\mathbb R_-^n\times (0,\infty)}(x,t)\in{\rm TMO_{\Delta_{N_-}}}(\Real_-^{n}\times (0,\infty))$ with
\begin{eqnarray*}\label{e1.8}
\|u\|_{{\rm TMO_{\Delta_{N}}}(\Real_+^{n+1})}&=&\max\Big\{ \|u_+\|_{{\rm TMO_{\Delta_{N_+}}}(\Real_+^{n}\times(0,\infty))},\ \|u_-\|_{{\rm TMO_{\Delta_{N_-}}}(\Real_-^{n}\times(0,\infty))}\Big\}<\infty.
\end{eqnarray*}}
\end{defn}

\begin{thm}\label{th1.1}
 \begin{itemize}\item[(1)]
If $f\in {\rm BMO_{\Delta_{N}}}(\RR)$, then  the function $u(x,t)=e^{t \Delta_{N}}f(x)\in {\rm TMO_{\Delta_{N}}}(\Real_+^{n+1})$ with
 $$ \|u\|_{{\rm TMO_{\Delta_{N}}}(\Real_+^{n+1})}\approx \|f\|_{{\rm BMO_{\Delta_{N}}}(\RR)}.$$
\item[(2)]    If $u\in {\rm TMO_{\Delta_{N}}}(\Real_+^{n+1})$, then    there exists some $f\in {\rm BMO_{\Delta_{N}}}(\Real^{n})$ such that $u(x,t)=e^{t \Delta_{N}}f(x)$,
and
$$
\|f\|_{{\rm BMO_{\Delta_{N}}}(\RR)}\leq C\|u\|_{{\rm TMO_{\Delta_{N}}}(\Real_+^{n+1})}
$$ with some constant $C>0$ independent of $u$ and $f$.

\item[(3)] We say that a tempered
distribution $f\in \rm BMO^{-1}_{\Delta_{N}}$
 if $$ \sup_{B(x_B, r_B)\subset \mathbb{R}^n\times (0,\infty)}     r_B^{-n} \int_0^{r_B^2}\int_{B(x_B, r_B)}  \left| e^{t \Delta_{N}}f(x) \right|^2
{dx dt }  <\infty.$$ According the definition above, it follows that a tempered distribution $f \in\mathbb{R}^{n}$ belongs to $\nabla\cdot (\rm BMO_{\Delta_{N}}(\mathbb{R}^{n}))^{n} $ if and only if there are $f_{j }\in \rm BMO_{\Delta_{N}}(\mathbb{R}^{n})$
 such that $f=\sum_{j=1}^{j=n}\partial_{j}f_{j}$. That is,
$$ \nabla\cdot (\rm BMO_{\Delta_{N}}(\mathbb{R}^{n}))^{n}=
 \rm BMO^{-1}_{\Delta_{N}}(\mathbb{R}^{n}).$$
 \end{itemize}
\end{thm}
As an application of Theorem  \ref{th1.1}, we shall consider the well-posedness of parabolic equations of Navier-Stokes type with the Neumann boundary condition, which borrows from \cite{AF1} completely. That is, for a.e. $x\in\mathbb{R}^{n}$, consider the equation
\begin{equation}\label{ee1.8}
\left\{
\begin{aligned}
\partial_{t}u-\Delta u=div_{x}f(u^{2}(t,x))&,\ \ \ \  \ \ \ \  x\in \Omega, \ t>0, \\
u(x,0)= u_{0}(x)&, \ \ \ \ \ \ \ \ x\in \Omega,\\
\frac{\partial u(x,t)}{\partial x_{n}}= 0&, \ \ \ \ \ \ \ \ x_{n}=0,\\
\end{aligned}
\right.
\end{equation}
where $\Omega=\mathbb{R}^{n}_{+}\cup \mathbb{R}^{n}_{-}$, we assume that $f: \mathbb{R}\rightarrow \mathbb{R}^{n}$
is globally Lipschitz continuous, and satisfies $$|f(x)|\leq C|x|,\, x\in\mathbb{R}.$$
Without loss of generality, we assume that $f(x)=x$ in this paper. The mild solutions of the system \eqref{ee1.8} is
\begin{align}\label{Aee1.8}
u(., t ) =e^{t\Delta_{N}}u_{0}-\int_{0}^{t}e^{(t-s)\Delta_{N}}div_{x} (u^{2}(s, . )) ds.
\end{align}
Clearly, from Theorem  \ref{th1.1}, we know that the divergence of a vector field with components in $\rm BMO_{\Delta_{N}}(\mathbb{R}^{n})$ is in  $\rm BMO^{-1}_{\Delta_{N}}(\mathbb{R}^{n})$.
To make sense to the free evolution term $e^{t\Delta_{N}}u_{0}$ from \eqref{Aee1.8}, recall that
in the case of the system \eqref{Aee1.8} (with the Neumann Laplacian in the background), the adapted
value space consists of divergence free elements $u_{0}$ in $\rm BMO^{-1}_{\Delta_{N}}(\mathbb{R}^{n})$ and is characterized by
$e^{t\Delta_{N}}u_{0}$ in the path space. Thus, we let
$$\varepsilon:=\left\{ u(x,t)\ \, measurables\, in \,\ \ \mathbb{R}^{n}\times (0, \infty) :
\left\|u\right\|_{\varepsilon}<\infty\right\},$$
with

\begin{eqnarray}\label{eee1.9}
\left\|u\right\|_{\varepsilon}= \left\|t^{1/2}u\right\|_{L^{\infty}(\mathbb{R}^{n}\times (0, \infty))}+
\left\|u\right\|_{T^{\infty,2}(\mathbb{R}^{n+1}_{+})},
\end{eqnarray}
where
$$\left\|u\right\|_{T^{\infty,2}(\mathbb{R}^{n+1}_{+})}=\sup_{x\in \mathbb{R}^{n}, t\in(0, \infty)}\left( t^{-n/2} \int_0^{t}\int_{B(x, \sqrt{t})}  \left|u(y,s)\right|^2
{dy ds }\right)^{1/2}.$$
Our second result in this paper reads as follows.

 \begin{thm}\label{th1.2}
Let $u_{0}\in {\rm BMO_{\Delta_{N}}^{-1}}(\RR)$. There exists
 $\epsilon>0$, such that, if $\|u_{0}\|_{\rm BMO_{\Delta_{N}}^{-1}}<\epsilon$, then the system \eqref{ee1.8} has a global mild solution $u\in\varepsilon$, which is unique one in the closed all
 $\{u\in\varepsilon: \|u\|_{\varepsilon}\leq2\epsilon\}$.
\end{thm}

This article is organized as follows. In Section 2, we give the proof of our   characterization theorem, Theorem \ref{th1.1}.
In Section 4, we give the proof of Theorem~\ref{th1.2} by using Theorem \ref{th1.1} as an application.

Throughout the article, the letters ``$c$ " and ``$C$ " will  denote (possibly different) constants
which are independent of the essential variables.

\vskip 1cm

\section{Proof of characterization theorem}
In this section, we will give the proof of Theorem  \ref{th1.1}.
\begin{proof}[Proof of Theorem \ref{th1.1}]
$(1)$
Since $f\in {\rm BMO}_{\Delta_N}(\RR)$, $f_{+,e}\in {\rm BMO}(\RR)$ and $f_{-,e}\in {\rm BMO}(\RR)$. And we have $$e^{t\Delta_N}f=(e^{t\Delta_N}f)_++(e^{t\Delta_N}f)_-=e^{t\Delta_{N_+}}f_++e^{t\Delta_{N_-}}f_-.$$
Hence, letting  $B=B(x_B, r_B)$,
\begin{align*}
 \int_0^{r_B^2}\int_B \abs{\nabla_{x}e^{t{ \Delta_N}}f(x) }^2{dxdt}
 &\le  2\int_0^{r_B^2}\int_{B} \abs{\nabla_{x}e^{t {\Delta_{N_+}}}f_+(x) }^2{dxdt}+2\int_0^{r_B^2}\int_{B} \abs{\nabla_{x}e^{t {\Delta_{N_-}}}f_-(x) }^2{dxdt}\\
&= 2\int_0^{r_B^2}\int_{B} \abs{\nabla_{x}e^{t {\Delta}}f_{+,e}(x) }^2{dxdt}+2\int_0^{r_B^2}\int_{B} \abs{\nabla_{x}e^{t {\Delta}}f_{-,e}(x) }^2{dxdt}\\
&\le C|B|\norm{f_{+,e}}_{{\rm BMO}(\RR)}^2+C|B|\norm{f_{-,e}}_{{\rm BMO}(\RR)}^2\\
&\le C|B|\norm{f}_{{\rm BMO_{\Delta_N}}(\RR)}^2,
\end{align*}
where we have used the result in \cite{FN1}.

 $(2)$ If $u\in {\rm TMO_{\Delta_{N}}}(\Real_+^{n+1})$, then $u_+\in {\rm TMO_{\Delta_{N_+}}(\Real_+^{n}\times(0,\infty))}$ and $u_-\in {\rm TMO_{\Delta_{N_-}}(\Real_-^{n}\times(0,\infty))}$. Since $u_+\in {\rm TMO_{\Delta_{N_+}}(\Real_+^{n}\times(0,\infty))}$,
 letting $u_{+,e}$ be the even extension of $u_+$ on $\mathbb R^{n+1}_+$. Then, by the result in \cite{FN1}, there exists an even function $f^1\in {\rm BMO}(\RR)$ such that $ u_{+,e}(x,t)=e^{t \Delta}f^1(x)$ for  $\ x\in \RR,$  and $\norm{f^1}_{\rm BMO(\RR)}\le C \| u_{+,e}\|_{{\rm TMO_{\Delta}}(\Real_+^{n+1})}.$ Then $f^1_+(x)\in \rm BMO_r(\mathbb R_+^n)$ and,
 for any   $x\in \RR_+$, $$u(x,t)= u_+(x,t)=u_{+,e}(x,t)=e^{t \Delta}f^1(x)=e^{t \Delta}f^1_{+,e}(x)=e^{t \Delta_{\rm N_+}}f^1_+(x)$$
  with $$\norm{f^1_+}_{{\rm BMO_r}(\mathbb R_+^n)}\le \norm{f^1}_{{\rm BMO}(\mathbb R^n)}\le  C \| u_{+,e}\|_{{\rm TMO_{\Delta}}(\Real_+^{n+1})}\le C \| u\|_{{\rm TMO_{\Delta_{N_+}}}(\Real_+^{n}\times(0,\infty))}$$  and $$\partial_{x_n} e^{t \Delta_{\rm N_+}}f^1_+(x',0)=0.$$

  Also, since $u_-\in {\rm TMO_{\Delta_{N_-}}(\Real_-^{n}\times(0,\infty))}$,
 letting $u_{-,e}$ be the even extension of $u_-$ on $\mathbb R^{n+1}_-$. Then, by the result in \cite{FN1}, there exists an even function $\bar f^1\in {\rm BMO}(\RR)$ such that $ u_{-,e}(x,t)=e^{t \Delta}\bar f^1(x)$ for  $\ x\in \RR,$  and $\norm{\bar f^1}_{\rm BMO(\RR)}\le C \| u_{-,e}\|_{{\rm TMO_{\Delta}}(\Real_+^{n+1})}.$ Then $\bar f^1_-(x)\in \rm BMO_r(\mathbb R_-^n)$ and,
 for any   $x\in \RR_-$, $$u(x,t)= u_-(x,t)=u_{-,e}(x,t)=e^{t \Delta}\bar f^1(x)=e^{t \Delta}\bar f^1_{-,e}(x)=e^{t \Delta_{\rm N_-}}\bar f^1_-(x)$$
  with $$\norm{\bar f^1_-}_{{\rm BMO_r}(\mathbb R_-^n)}\le \norm{\bar f^1}_{{\rm BMO}(\mathbb R^n)}\le  C \| u_{-,e}\|_{{\rm TMO_{\Delta}}(\Real_+^{n+1})}\le C \| u\|_{{\rm TMO_{\Delta_{N_-}}}(\Real_-^{n}\times(0,\infty))}$$ and $$\partial_{x_n} e^{t \Delta_{\rm N_-}}\bar f^1_-(x',0)=0.$$

  Let $f(x)=f^1_+(x)+\bar f^1_-(x),$ for any $x\in \RR.$ Then $f\in \rm BMO_{\Delta_N}(\RR)$ and
  $$e^{t \Delta_{N}}f(x)=(e^{t \Delta_{N}}f)_+(x)+(e^{t \Delta_{N}}f)_-(x)=e^{t \Delta_{N_+}}f^1_+(x)+e^{t \Delta_{N_-}}\bar f^1_-(x)=u(x,t)$$
  for any $x\in \RR.$ And we have $$\norm{f}_{{\rm BMO_{\Delta_N}}(\mathbb R^n)}\le \max \Big\{ \norm{f^1}_{{\rm BMO_r}(\mathbb R^n_+)}, \norm{\bar f^1}_{{\rm BMO_r}(\mathbb R^n_-)}\Big\}\le  C \|u\|_{{\rm TMO_{\Delta_{N}}}(\Real_+^{n+1})}$$ and $$\partial_{x_n} e^{t \Delta_{\rm N}} f(x',0)=\partial_{x_n} e^{t \Delta_{\rm N_+}} f^1_+(x',0)+\partial_{x_n} e^{t \Delta_{\rm N_-}} \bar f^1_-(x',0)=0.$$

  $(3)$
  For any $f\in \rm\nabla\cdot \rm BMO_{\Delta_{N}}(\mathbb{R}^{n})$, there exists
   $f_{1},f_{2},\cdot\cdot\cdot, f_{n}\in \rm BMO_{\Delta_{N}}(\mathbb{R}^{d})$  such that $f=\sum_{j=1}^{j=n}\partial_{j}f_{j}$, we have
   $$\left\|f\right\|_{\rm BMO_{\Delta_{N}}^{-1}(\mathbb{R}^{n})}\leq
   \sum_{j=1}^{j=n}\left\|\partial_{j}f_{j}\right\|_{\rm BMO_{\Delta_{N}}^{-1}(\mathbb{R}^{n})}
   \leq
   \sum_{j=1}^{j=n}\left\|f_{j}\right\|_{\rm BMO_{\Delta_{N}}
   (\mathbb{R}^{n})}.$$
   On the other hand, if  $f\in \rm BMO_{\Delta_{N}}^{-1}(\mathbb{R}^{n})$ and $f_{j,k}=\partial_{j}\partial_{k}(-\Delta)^{-1}f$,  It is suffice to prove we
    $f_{j,k}\in \rm BMO_{\Delta_{N}}^{-1}(\mathbb{R}^{n})$. In fact,
    by  the term \eqref{e2.2} below, we have that
 \begin{align*}
\left\|\partial_{j}\partial_{k}(-\Delta_{N})^{-1}f\right
\|_{\rm BMO_{\Delta_{N}}^{-1}(\mathbb{R}^{d})}
&\le\frac{1}{2}\left\|\partial_{j}\partial_{k}(-\Delta_{N})^{-1}f_{+,e}\right\|_{\rm BMO^{-1}(\mathbb{R}^{n})}^{2}+
\frac{1}{2}\left\|\partial_{j}\partial_{k}(-\Delta_{N})^{-1}
f_{-,e}\right\|_{\rm BMO^{-1}(\mathbb{R}^{n})}^{2}\\
&\le\frac{1}{2}\left\|f_{+,e}\right\|_{\rm BMO^{-1}(\mathbb{R}^{n})}^{2}+
\frac{1}{2}\left\|
f_{-,e}\right\|_{\rm BMO^{-1}(\mathbb{R}^{n})}^{2}\lesssim\left\|
f\right\|_{\rm BMO_{\Delta_{N}}^{-1}(\mathbb{R}^{n})}^{2}\\
\end{align*}
Thus we have $f_{k}=-\partial_{k}(-\Delta)^{-1}f\in \rm BMO_{\Delta_{N}}(\mathbb{R}^{n})$ and
$f=\sum_{k=1}^{k=n}\partial_{k}f_{k}$.
 And we end the proof.
 \end{proof}

\section{Proof of  well-posedness of parabolic equations  with  Neumann boundary condition
\setcounter{equation}{0}
}
In this section, we shall prove the well-posedness of the system \eqref{ee1.8}
 with initial data in $\rm BMO_{\Delta_{N}}^{-1}$, which can be obtained by combining the characterization of the $\rm BMO_{\Delta_{N}}(\mathbb{R}^{n})$
 derived in Theorem  \ref{th1.1} and the Banach contraction mapping principle. Therefore
it is necessary to examine the linear and nonlinear terms of  \eqref{ee1.8}. To do so, we need the following functions estimates.

\begin{prop}\label{le2.33} According to  all the notations as in \eqref{eee1.9}, one has
\begin{equation}\label{ee2.1}
\left\|f\right\|_{T^{\infty, 2}(\mathbb{R}^{n+1}_{+})}\thickapprox \left\|f_{+,e}\right\|_{T^{\infty, 2}(\mathbb{R}^{n+1}_{+})}+
\left\|f_{-,e}\right\|_{T^{\infty, 2}(\mathbb{R}^{n+1}_{+})},
\end{equation}
here $f_{\pm,e}$ is the even extension of the restriction of $f$ from $\mathbb{R}^{n}_{\pm}$. Namely, $f\in T^{\infty, 2}(\mathbb{R}^{n+1}_{+})$
if and only if $f_{+,e}\in T^{\infty, 2}(\mathbb{R}^{n+1}_{+})$ and $f_{-,e}\in T^{\infty, 2}(\mathbb{R}^{n+1}_{+})$.  Simlarly, we also have
\begin{equation}\label{e2.2}
\left\|f\right\|_{\rm BMO^{-1}_{\Delta_{N}}(\mathbb{R}^{n})}\thickapprox \left\|f_{+,e}\right\|_{\rm BMO^{-1}(\mathbb{R}^{n})}+
\left\|f_{-,e}\right\|_{\rm BMO^{-1}(\mathbb{R}^{n})},
\end{equation}
\begin{equation}\label{ee2.3}
\left\|f\right\|_{T^{\infty, 1}(\mathbb{R}^{n+1}_{+})}\thickapprox \left\|f_{+,e}\right\|_{T^{\infty, 1}(\mathbb{R}^{n+1}_{+})}+
\left\|f_{-,e}\right\|_{T^{\infty, 1}(\mathbb{R}^{n+1}_{+})},
\end{equation}
\begin{equation}\label{ee2.4}
\left\|t^{1/2}f\right\|_{L^{\infty}((0, \infty)\times \mathbb{R}^{n})}\thickapprox \left\|t^{1/2}f_{+,e}\right\|_{L^{\infty}((0, T)\times \mathbb{R}^{n})}+
\left\|t^{1/2}f_{-,e}\right\|_{L^{\infty}((0, \infty)\times \mathbb{R}^{n})},
\end{equation}
and there exist a positive constant $C$ such that
\begin{equation}\label{ee2.5}
\left\|t^{1/2}\exp(t\Delta_{N})f\right\|_{L^{\infty}((0, \infty)\times \mathbb{R}^{n})}\leq C \left\|f_{+,e}\right\|_{\dot{B}_{\infty,\infty}^{-1}(\mathbb{R}^{n})}+
C\left\|f_{-,e}\right\|_{\dot{B}_{\infty,\infty}^{-1}(\mathbb{R}^{n})}.
\end{equation}
\end{prop}
\begin{proof}
 Let us begin with the proof of the inequality \eqref{ee2.1},  for any $t>0$,
\begin{equation}\label{ee2.6}
\begin{aligned}
& t^{-n/2} \int_0^{t}\int_{B(x, \sqrt{t})}  \left|f(y,s)\right|^2
{dy ds }\\
 &\ \ \ \le t^{-n/2} \int_0^{t}\int_{B(x, \sqrt{t}), y\in\mathbb{R}^{n}_{+}}  \left|f_{+,e}(y,s)\right|^2
{dy ds }+ t^{-n/2} \int_0^{t}\int_{B(x, \sqrt{t}), y\in\mathbb{R}^{n}_{-}}  \left|f_{-,e}(y,s)\right|^2
{dy ds }\\
&\ \ \ \le t^{-n/2} \int_0^{t}\int_{B(x, \sqrt{t})}  \left|f_{+,e}(y,s)\right|^2
{dy ds }+ t^{-n/2} \int_0^{t}\int_{B(x, \sqrt{t})}  \left|f_{-,e}(y,s)\right|^2
{dy ds }\\
&\ \ \ \le \left\|f_{+,e}\right\|_{T^{\infty, 2}(\mathbb{R}^{n+1}_{+})}^{2}+
\left\|f_{-,e}\right\|_{T^{\infty, 2}(\mathbb{R}^{n+1}_{+})}^{2}.
\end{aligned}
\end{equation}
 Conversely,
  \begin{align*}
& t^{-n/2} \int_0^{t}\int_{B(x, \sqrt{t})}  \left|f_{+,e}(y,s)\right|^2
{dy ds }\\
&\ \ \ \le t^{-n/2} \int_0^{t}\int_{B(x, \sqrt{t}), y\in\mathbb{R}^{n}_{+}}  \left|f_{+}(y,s)\right|^2
{dy ds }+ t^{-n/2} \int_0^{t}\int_{B(\tilde{x}, \sqrt{t}), y\in\mathbb{R}^{n}_{+}}  \left|f_{+}(y,s)\right|^2
{dy ds }\\
&\ \ \ \le 2\left\|f\right\|_{T^{\infty, 2}(\mathbb{R}^{n+1}_{+})}^{2},
\end{align*}
and, similarly,
 \begin{align*}
 t^{-n/2} \int_0^{t}\int_{B(x, \sqrt{t})}  \left|f_{-,e}(y,s)\right|^2
{dy ds }
\le 2\left\|f\right\|_{T^{\infty, 2}(\mathbb{R}^{n+1}_{+})}^{2}.
\end{align*}
These together with the estimates \eqref{ee2.6} imply
 \begin{align*}
\left\|f\right\|_{T^{\infty, 2}(\mathbb{R}^{n+1}_{+})}\le
\left\|f_{+,e}\right\|_{T^{\infty, 2}(\mathbb{R}^{n+1}_{+})}+
\left\|f_{-,e}\right\|_{T^{\infty, 2}(\mathbb{R}^{n+1}_{+})}\le 2\sqrt{2}\left\|f\right\|_{T^{\infty, 2}(\mathbb{R}^{n+1}_{+})},
\end{align*} and which completes the proof of \eqref{ee2.1}.

For the term \eqref{e2.2}, obviously,
 \begin{align*}
& t^{-n/2} \int_0^{t}\int_{B(x, \sqrt{t})}  \left|\exp(t\Delta_{N})f(y,s)\right|^2
{dy ds }\\
 &\ \ \ \ \le \frac{1}{2}t^{-n/2} \int_0^{t}\int_{B(x, \sqrt{t})}  \left|\exp(t\Delta)f_{+,e}(y,s)\right|^2
{dy ds }\\
&\ \ \ \ \ \ \ \ +\frac{1}{2}t^{-n/2} \int_0^{t}\int_{B(x, \sqrt{t})}  \left|\exp(t\Delta)f_{-,e}(y,s)\right|^2
{dy ds }\\
&\ \ \ \ \le \frac{1}{2}\left\|f_{+,e}\right\|_{\rm BMO^{-1}(\mathbb{R}^{n})}^{2}+
\frac{1}{2}\left\|f_{-,e}\right\|_{\rm BMO^{-1}(\mathbb{R}^{n})}^{2}.
\end{align*}

Similarly, we have
\begin{align*}
& t^{-n/2} \int_0^{t}\int_{B(x, \sqrt{t})}  \left|\exp(t\Delta)f_{+,e}(y,s)\right|^2
{dy ds }\\
&\ \ \ \ \le 2t^{-n/2} \int_0^{t}\int_{B(x, \sqrt{t})}  \left|\exp(t\Delta_{N})f(y,s)\right|^2
{dy ds }\le2\left\|f\right\|_{\rm BMO^{-1}_{\Delta_{N}}(\mathbb{R}^{n})}^{2}
\end{align*}
and
\begin{align*}
t^{-n/2} \int_0^{t}\int_{B(x, \sqrt{t})}  \left|\exp(t\Delta)f_{-,e}(y,s)\right|^2
{dy ds }
 \le2\left\|f\right\|_{\rm BMO^{-1}_{\Delta_{N}}(\mathbb{R}^{n})}^{2}.
\end{align*}
Finally, we have equivalent norms
 \begin{align*}
\frac{\sqrt{2}}{4}\left\|f\right\|_{\rm BMO^{-1}_{\Delta_{N}}(\mathbb{R}^{n})}\le
\left\|f_{+,e}\right\|_{\rm BMO^{-1}(\mathbb{R}^{n})}+
\left\|f_{-,e}\right\|_{\rm BMO^{-1}(\mathbb{R}^{n})}\le \frac{\sqrt{2}}{2}
\left\|f\right\|_{\rm BMO^{-1}_{\Delta_{N}}(\mathbb{R}^{n})}.
\end{align*}
Collecting terms above, we have the desired estimate \eqref{e2.2} and thus complete the proof of \eqref{e2.2}.
The terms \eqref{ee2.3}-\eqref{ee2.5} can be proved similarly, here we omit the details.
 \end{proof}

\begin{prop} \label{le3332.2}
Suppose $u_{0}\in \rm BMO^{-1}_{\Delta_{N}}(\mathbb{R}^{n})$, then we have that $$\left\|\exp(t\triangle_{N})u_{0}\right\|_{\varepsilon}\lesssim
\left\|u_{0}\right\|_{\rm BMO^{-1}_{\Delta_{N}}(\mathbb{R}^{n})}.$$
\end{prop}
\begin{proof}
By Lemma~\ref{le2.33} and $\rm  BMO^{-1}(\mathbb{R}^{n})
\hookrightarrow \dot{B}^{-1}_{\infty,\infty}(\mathbb{R}^{n})$, where
$\dot{B}^{-1}_{\infty,\infty}(\mathbb{R}^{n})$ is homogeneous Besov space (the definition of homogeneous Besov spaces, see \cite{YSY}),
it thus follows that
\begin{align*}\label{ee1.9}
\left\|\exp(t\triangle_{N})u_{0}\right\|_{\varepsilon}
&=\left\|t^{1/2}\exp(t\triangle_{N})u_{0}\right\|_{L^{\infty}((0, \infty)\times \mathbb{R}^{n})}\\
&\ \ \ \ \ +\sup_{x\in \mathbb{R}^{n}, t\in(0, \infty)}\left( t^{-n/2} \int_0^{t}\int_{B(x, \sqrt{t})}  \left|\exp(s\triangle_{N})u_{0}(y,s)\right|^2
{dy ds }\right)^{1/2} \\
&\leq \left\|u_{0,+,e}\right\|_{\dot{B}_{\infty,\infty}^{-1}(\mathbb{R}^{n})}+
\left\|u_{0,-,e}\right\|_{\dot{B}_{\infty,\infty}^{-1}(\mathbb{R}^{n})}
+\left\|u_{0}\right\|_{BMO^{-1}_{\Delta_{N}}(\mathbb{R}^{n})}\lesssim \left\|u_{0}\right\|_{BMO^{-1}_{\Delta_{N}}(\mathbb{R}^{n})}.
\end{align*}
Collecting terms, we have the desired estimate and thus complete the proof of  Proposition \ref{le3332.2}.
\end{proof}
Let $\alpha:= u\times v$ and $u,v\in\varepsilon$ and
$$\mathcal{A}(\alpha)(t):=\int_{0}^{t}e^{(t-s)\Delta_{N}}div_{x}\alpha(s, . ) ds.$$
\begin{prop} \label{leq2.2} Let $u, v\in \varepsilon$. A positive constant $C$ exists such that
 \begin{align*}
\left\|t^{\frac{1}{2}}\mathcal{A}(\alpha)\right\|_{L^{\infty}((0, \infty)\times \mathbb{R}^{n})}
\leq C\left\|\alpha\right\|_{T^{\infty,1}(\mathbb{R}^{n})}+C\left\|s\alpha(s)\right\|_{L^{\infty}((0, \infty)\times \mathbb{R}^{n})}.
\end{align*}
\end{prop}
\begin{proof}
We first split the integral as
\begin{align*}
\mathcal{A}(\alpha)(t):=\int_{0}^{t/2}e^{(t-s)\Delta_{N}}div_{x}\alpha(s, . ) ds+
\int_{t/2}^{t}e^{(t-s)\Delta_{N}}div_{x}\alpha(s, . ) ds
\end{align*}
From the definition of $p_{t-s,\Delta_{N}}(x,y)$, using the decay of the heat kernel at $\infty$, the divergence term can be estimated as following
\begin{align*}
&\left|\nabla\exp((t-s)\triangle_{N})\alpha(s,x)\right|\\
&\ \ \ \ \lesssim\int_{\mathbb{R}^{n}_{+}}\left|\nabla p_{t-s,\Delta_{N}}(x,y)\right|
\left|\alpha_{+}(s,y)\right|dy\\
&\ \ \ \ \lesssim\int_{\mathbb{R}^{n}_{+}}\left(\sum_{i=1}^{n-1} \left|\frac{\partial}{\partial x_{i}}p_{t-s,\Delta_{N_{+}}}(x,y)\right|
+\left|\frac{\partial}{\partial x_{n}}p_{t-s,\Delta_{N_{+}}}(x,y)\right|\right)
\left|\alpha_{+}(s,y)\right|dy\\
&\ \ \ \ \lesssim\int_{\mathbb{R}^{n}_{+}}\sum_{i=1}^{n} \left|\frac{x_{i}-y_{i}}{2t}\frac{1}{(4\pi t)^{n/2}}e^{-\frac{|x-y|^{2}}{4t}}\right|
\left|\alpha_{+}(s,y)\right|dy\\
&\ \ \ \ \ \ \ \ +\int_{\mathbb{R}^{n}_{+}}e^{-\frac{|x_{n}+y_{n}|^{2}}{4t}}\left(\sum_{i=1}^{n-1} \left|\frac{x_{i}-y_{i}}{2t}\frac{1}{(4\pi t)^{n/2}}e^{-\frac{|x'-y'|^{2}}{4t}}
\right|+\left|\frac{x_{n}+y_{n}}{2t}\frac{1}{(4\pi t)^{n/2}}e^{-\frac{|x'-y'|^{2}}{4t}}
\right|\right)
\left|\alpha_{+}(s,y)\right|dy\\
&\ \ \ \ \lesssim\int_{\mathbb{R}^{n}}\left|\nabla p_{t-s,\Delta}(x,y)\right|
\left|\alpha_{+,e}(s,y)\right|dy\\
&\ \ \ \ \lesssim\int_{\mathbb{R}^{n}}\left(t-s\right)^{-n/2}\left(1+\left(t-s\right)^{-1/2}\left|x-y\right| \right)^{-n-1}
\left|\alpha_{+,e}(s,y)\right|dy.
\end{align*}
We can also obtain similar estimates  for any $x\in\mathbb{R}^{n}_{-}$,
\begin{align*}
\left|\nabla\exp(t\triangle_{N})\alpha(s,x)\right|
&\lesssim\int_{\mathbb{R}^{n}}\left|\nabla p_{t-s,\Delta}(x,y)\right|
\left|\alpha_{-,e}(s,y)\right|dy\\
&\lesssim\int_{\mathbb{R}^{n}}\left(t-s\right)^{-n/2}\left(1+\left(t-s\right)^{-1/2}\left|x-y\right| \right)^{-n-1}
\left|\alpha_{-,e}(s,y)\right|dy.
\end{align*}
Since $0 < s < t/2$ implies $t- s \thicksim t$, using the fact that every cube of side length $\sqrt{ t}$ is contained in a ball of radius $\sqrt{t}$, we find
that
\begin{align*}
&\sup_{t>0}\left|t^{\frac{1}{2}}\int_{0}^{t/2}e^{(t-s)\Delta_{N}}div_{x}\alpha(s, x) ds\right|\\
& \ \ \ \ \ \leq \sup_{t>0,x\in\mathbb{R}^{n}}\left|t^{\frac{1}{2}}\int_{0}^{t/2}\int_{\mathbb{R}^{n}}\left|\nabla p_{t-s,\Delta}(x,y)\right|\left|\alpha_{+,e}(s,y)\right|dy
ds\right|\\
& \ \ \ \ \ \ \ \ \ + \sup_{t>0,x\in\mathbb{R}^{n}}\left|t^{\frac{1}{2}}\int_{0}^{t/2}
\int_{\mathbb{R}^{n}}\left|\nabla p_{t-s,\Delta}(x,y)\right|
\left|\alpha_{-,e}(s,y)\right|dyds\right|\\
& \ \ \ \ \ \leq \sup_{t>0,x\in\mathbb{R}^{n}}\left|t^{\frac{1}{2}}\int_{0}^{t/2}
\sum_{k\in\mathbb{Z}^{n}}t^{-n/2}\left(1+|k|\right)^{-n}
\left|\alpha_{+,e}(s,y)\right|dy
ds\right|\\
& \ \ \ \ \  \ \ \ \ \ +\sup_{t>0,x\in\mathbb{R}^{n}}\left|t^{\frac{1}{2}}\int_{0}^{t/2}
\sum_{k\in\mathbb{Z}^{n}}t^{-n/2}\left(1+|k|\right)^{-n}
\left|\alpha_{-,e}(s,y)\right|dy
ds\right|\\
& \ \ \ \ \ \leq \sum_{k\in\mathbb{Z}^{n}}\left(1+|k|\right)^{-n}
\left(\|\alpha_{+,e}\|_{T^{\infty, 1}(\mathbb{R}^{n+1}_{+})}+\|\alpha_{-,e}\|_{T^{\infty, 1}(\mathbb{R}^{n+1}_{+})}\right).
\end{align*}
Now we look at $t/2 < s <t$, the estimates of  the kernel of the heat semigroup imply that
\begin{align*}
&\sup_{t>0}\left|t^{\frac{1}{2}}\int_{t/2}^{t}e^{(t-s)\Delta_{N}}div_{x}\alpha(s, x) ds\right|\\
& \ \ \ \ \ \leq \sup_{t>0,x\in\mathbb{R}^{n}}\left|t^{\frac{1}{2}}\int_{0}^{t/2}\int_{\mathbb{R}^{n}}\left|\nabla p_{t-s,\Delta}(x,y)\right|\left|\alpha_{+,e}(s,y)\right|dy
ds\right|\\
& \ \ \ \ \ \ \ \ \ + \sup_{t>0,x\in\mathbb{R}^{n}}\left|t^{\frac{1}{2}}\int_{0}^{t/2}
\int_{\mathbb{R}^{n}}\left|\nabla p_{t-s,\Delta}(x,y)\right|
\left|\alpha_{-,e}(s,y)\right|dyds\right|\\
& \ \ \ \ \ \leq \sup_{t>0,x\in\mathbb{R}^{n}}\left|t^{\frac{1}{2}}\int_{t/2}^{t}
\frac{1}{(4\pi(t-s))^{n/2}}\left(
\int_{\mathbb{R}^{n}}e^{-\frac{|x-y|^{2}}{4(t-s)}}\frac{|x-y|}{2s(t-s)}dy\right)
\left\|s\alpha_{+,e}(s,y)\right\|_{L^{\infty}((0, \infty)\times \mathbb{R}^{n})}ds\right|\\
& \ \ \ \ \  \ \ \ \ \ +\sup_{t>0,x\in\mathbb{R}^{n}}\left|t^{\frac{1}{2}}\int_{t/2}^{t}
\frac{1}{(4\pi(t-s))^{n/2}}\left(
\int_{\mathbb{R}^{n}}e^{-\frac{|x-y|^{2}}{4(t-s)}}\frac{|x-y|}{2s(t-s)}dy\right)
\left\|s\alpha_{-,e}(s,y)\right\|_{L^{\infty}((0, \infty)\times \mathbb{R}^{n})}ds\right|\\
& \ \ \ \ \ \leq \left\|s\alpha_{-,e}(s,y)\right\|_{L^{\infty}((0, \infty)\times \mathbb{R}^{n})}+\left\|s\alpha_{-,e}(s,y)\right\|_{L^{\infty}((0, \infty)\times \mathbb{R}^{n})}.
\end{align*}
Resuming the above estimates, the proof of Proposition \ref{leq2.2} is completed.
\end{proof}
To derive the Carleson measure estimate, we write
\begin{align*}
\mathcal{A}(\alpha)(t):
&=\int_{0}^{t}e^{-(t-s)\L}div_{x}\alpha(s, . ) ds\\
&=\int_{0}^{t}e^{-(t-s)\L}L\mathcal{T}\left(s^{1/2}\alpha(s, . )\right) ds +\int_{0}^{\infty}e^{-(t+s)\L}\left(div_{x}\alpha(s, . )\right) ds\\
&\ \ \ \ \  -\int_{t}^{\infty}e^{-(t-s)\L}s^{-1/2}\left(div_{x}s^{1/2}\alpha(s, . )\right) ds\\
&=:\mathcal{A}_{1}(\alpha)(t)+\mathcal{A}_{2}(\alpha)(t)
+\mathcal{A}_{3}(\alpha)(t),
\end{align*}
where $\L=-\Delta_{N}$ and $$\mathcal{T}F(s,x)=:(s\L)^{-1}(I-e^{-2s\L})div F(s,x).$$
Note that, the operator $\Delta_{N}$ satisfes Gaussian bound and has a bounded $H_{\infty}$-calculus in $L^{2}(\mathbb{R}^{n})$.

For the estimate on $\mathcal{A}_{1}$, we apply the following two lemmas. The first one is an extension of \cite[Theorem 3.2]{9}
 using the structure of the maximal regularity operator.

\begin{lem}\label{le2.34} The operator
\begin{equation*}
\mathcal{M}^{+}: T^{\infty,2}(\mathbb{R}^{n+1}_{+})\rightarrow T^{\infty,2}(\mathbb{R}^{n+1}_{+}),
\end{equation*}
\begin{equation*}
(\mathcal{M}^{+}F)(t,\cdot):=\int_{0}^{t}Le^{-(t-s)\L}F(s,\cdot)ds,
\end{equation*}
is bounded.
\end{lem}
\begin{proof}
According to  \cite[Lemma 1,19]{12}, $(t\L e^{t\L})_{t>0}$ satisfies Gaussian estimates,
therefore in particular the weaker $\L^{2}(\mathbb{R}^{n})$ off-diagonal estimates of  \cite[Theorem 3.2]{9}. Hence, we can
apply \cite[Theorem 3.2]{9} to obtain that $\mathcal{M}^{+} : T^{\infty,2}(\mathbb{R}^{n+1}_{+})\rightarrow T^{\infty,2}(\mathbb{R}^{n+1}_{+})$.
The proof of Lemma \ref{le2.34} is thus completed.
\end{proof}

\begin{lem}\label{le3.55} The operator
$$\mathcal{T}F(s,x):=T_{s}F(s,\cdot)(x), \  \ \forall (s,x)\in\mathbb{R}^{+}\times \mathbb{R}^{n}$$
with
\begin{equation*}
T_{s}=:(s\L)^{-1}(I-e^{-2s\L})div F(s,x)
\end{equation*}
is bounded from $T^{\infty,2}(\mathbb{R}^{n+1}_{+})$ to $T^{\infty,2}(\mathbb{R}^{n+1}_{+})$.
\end{lem}

\begin{proof}
Since
$$T_{s}=:-s^{-1/2}\int_{0}^{2s}e^{-\mu \L}div d\mu,$$
the estimate
 \begin{align*}
& t^{-n/2} \int_0^{t}\int_{B(x, \sqrt{t})}  \left|s^{-1/2}\int_{0}^{\infty}e^{-\mu \L}div d\mu f(y,s)\right|^2
{dy ds }\\
 &\ \ \ \ \ \ \ \le \frac{1}{2}\left(t^{-n/2} \int_0^{t}\int_{B(x, \sqrt{t})}  \left|s^{-1/2}\int_{0}^{2s}e^{-\mu\L}div d\mu f_{+,e}(y,s)\right|^2
{dy ds }\right)\\
&\ \ \ \ \ \ \ \ \ \ \ +\frac{1}{2}\left( t^{-n/2} \int_0^{t}\int_{B(x, \sqrt{t})}  \left|s^{-1/2}\int_{0}^{2s}e^{-\mu \L}div d\mu f_{-,e}(y,s)\right|^2
{dy ds }\right)\\
&\ \ \ \ \ \ \ \le \frac{1}{2}\left\|s^{-1/2}\int_{0}^{2s}e^{-\mu \L}div f_{+,e}d\mu \right\|_{T^{\infty,2}(\mathbb{R}^{n})}+
\frac{1}{2}\left\|s^{-1/2}\int_{0}^{2s}e^{-\mu \L}div f_{-,e}d\mu \right\|_{T^{\infty,2}(\mathbb{R}^{n+1}_{+})}\\
&\ \ \ \ \ \ \ \le\frac{1}{2}\left\| f_{+,e}\right\|_{T^{\infty,2}(\mathbb{R}^{n+1}_{+})}+\frac{1}{2}
\left\| f_{-,e} \right\|_{T^{\infty,2}(\mathbb{R}^{n+1}_{+})}
\end{align*}
holds true. Therefore, the desired estimate now follows  due to the Proposition \ref{le2.33}. We complete the proof of Lemma \ref{le3.55}.
\end{proof}
%
An application of Fubini¡¯s theorem yields
\begin{align*}
&\left<\mathcal{A}_{2}F, G\right>=\int_{\mathbb{R}^{n}}\int_{0}^{\infty}e^{(t+s)\Delta_{N}}div F(s,y)ds\cdot\overline{G(t,y)}dtdy\\
&\ \ \ \ \ \ \ =\int_{\mathbb{R}^{n}}\int_{0}^{\infty}F(s,y)\overline{\nabla e^{(t+s)\Delta_{N}}G(t,y)}dtdsdy\\
&\ \ \ \ \ \ \ =\left<F, \nabla e^{s\Delta_{N}}\int_{0}^{\infty}e^{t\Delta_{N}}G(t,\cdot)dt\right>=\left<F, \mathcal{A}_{2}^{*}G\right>.
\end{align*}

We thus have the following lemma.
\begin{lem} \label{lee2.4e}
There exists a positive constant $c>0$ such that
$$\left\|(\mathcal{A}_{2}^{*}G)(s,\cdot)\right
\|_{T^{1,\infty}(\mathbb{R}^{n+1}_{+})}
=:\left\|\nabla e^{s\Delta_{N}}\int_{0}^{\infty}e^{t\Delta_{N}}G(t,\cdot)dt
\right\|_{T^{1,\infty}(\mathbb{R}^{n+1}_{+})}\leq \left\|G\right\|_{T^{1,2}(\mathbb{R}^{n+1}_{+})}$$
\end{lem}
\begin{proof}
To derive the estimate of $\mathcal{A}_{2}^{*}G$   in the $T^{1,\infty}(\mathbb{R}^{n+1}_{+})$ norm, we employ the theory of Hardy spaces associated with operators $\Delta_{N}$. Recall that a  Hardy-type space associated to $\Delta_{N}$ was introduced in \cite{DDSY},  defined  by
\begin{equation}\label{e2.7}
 H^1_{\Delta_{N}}(\mathbb{R}^{n})=\big\{ f\in L^1(\mathbb{R}^{n}): S_{\Delta_{N}}(f)(x)\in L^1(\RR) \big\}
\end{equation}
in the norm of
$$
\left\|f\right\|_{H^1_{\Delta_{N}}(\mathbb{R}^{n})}=\left\| S_{\Delta_{N}}(f)\right\|_{L^1(\RR)},
$$
where $$S_{\Delta_{N}}(f)(x)=\left(\int_{0}^{\infty}\int_{|y-x|<t}
\left|t^{2}\Delta_{N}\exp(t^{2}\Delta_{N})f(y)\right|^{2}\frac{dydt}{t^{n+1}}
\right)^{\frac{1}{2}}.$$

By the characterization of $H^1_{\Delta_{N}}(\mathbb{R}^{n})$, we find that
 \begin{align*}
& t^{-n/2} \int_0^{t}\int_{B(x, \sqrt{t})}  \left|\nabla e^{s\Delta_{N}}\int_{0}^{\infty}e^{\theta\Delta_{N}}
G(\theta,\cdot)d\theta\right|
{dy ds }\\
 &\ \ \ \ \ \le \frac{1}{2}\left(t^{-n/2} \int_0^{t}\int_{B(x, \sqrt{t})}  \left|\nabla e^{s\Delta}\left[\int_{0}^{\infty}e^{\theta\Delta_{N}}
G(\theta,\cdot)d\theta\right]_{+,e}\right|
{dy ds }\right)\\
&\ \ \ \ \ \ \ \ +\frac{1}{2}\left(t^{-n/2} \int_0^{t}\int_{B(x, \sqrt{t})}  \left|\nabla e^{s\Delta}\left[\int_{0}^{\infty}e^{\theta\Delta_{N}}
G(\theta,\cdot)d\theta\right]_{-,e}\right|
{dy ds }\right)\\
&\ \ \ \ \ \le\frac{1}{2}\left\| \left[\nabla \int_{0}^{\infty}e^{\theta\Delta_{N}}G(\theta,\cdot)d\theta
\right]_{+,e}\right\|_{H^{1}(\mathbb{R}^{n})}+\frac{1}{2}\left\| \left[\nabla \int_{0}^{\infty}e^{\theta\Delta_{N}}G(\theta,\cdot)d\theta
\right]_{-,e}\right\|_{H^{1}(\mathbb{R}^{n})}\\
&\ \ \ \ \ \le\left\| \nabla \int_{0}^{\infty}e^{\theta\Delta_{N}}G(\theta,\cdot)d\theta
\right\|_{H^{1}_{\Delta_{N}}(\mathbb{R}^{n})}.
\end{align*}

From the definition of $H^1_{\Delta_{N}}(\mathbb{R}^{n})$, we have
\begin{align*}
&\left\| \nabla \int_{0}^{\infty}e^{\theta\Delta_{N}}G(\theta,\cdot)d\theta
\right\|_{H^{1}_{\Delta_{N}}(\mathbb{R}^{n})}\\
&=\int_{\mathbb{R}^{n}}\left|\left(\int_{0}^{\infty}\int_{|y-x|<t}
\left|t^{2}\Delta_{N}\exp(t^{2}\Delta_{N})\left(\nabla \int_{0}^{\infty}e^{\theta\Delta_{N}}G(\theta,y)d\theta\right)\right|^{2}\frac{dydt}{t^{n+1}}
\right)^{\frac{1}{2}}\left(x\right)\right|dx\\
 &\le \frac{1}{2}\left(t^{-n/2} \int_0^{t}\int_{B(x, \sqrt{t})}  \left|\nabla e^{s\Delta}\left[\int_{0}^{\infty}e^{\theta\Delta_{N}}
G(\theta,\cdot)d\theta\right]_{+,e}\right|
{dy ds }\right)\\
&\ \ \ \ \ \ \ +\frac{1}{2}\left(t^{-n/2} \int_0^{t}\int_{B(x, \sqrt{t})}  \left|\nabla e^{-s\Delta}\left[\int_{0}^{\infty}e^{\theta\Delta_{N}}
G(\theta,\cdot)d\theta\right]_{-,e}\right|
{dy ds }\right)\\
&\lesssim\left\| \left[\nabla \int_{0}^{\infty}e^{\theta\Delta}G_{+,e}(\theta,\cdot)d\theta
\right]\right\|_{H^{1}(\mathbb{R}^{n})}+\left\| \left[\nabla \int_{0}^{\infty}e^{\theta\Delta}G_{+,e}(\theta,\cdot)d\theta
\right]\right\|_{H^{1}(\mathbb{R}^{n})}\\
&\lesssim \left\|G_{+,e}\right\|_{T^{1,2}(\mathbb{R}^{n+1}_{+})}
+\left\|G_{-,e}\right\|_{T^{1,2}(\mathbb{R}^{n+1}_{+})}.
\end{align*}
We complete the proof of Lemma \ref{lee2.4e}.
\end{proof}

\begin{lem}\label{leee2.e5}
 Let $$RF(\cdot, s)=\int_{s}^{\infty} e^{(s+\tau)\Delta_{N}}\tau^{-\frac{1}{2}}div F(\cdot, \tau)d\tau.$$ Then we have
 $$\left\|RF(\cdot, t)\right\|_{T^{\infty,2}(\mathbb{R}^{n+1}_{+})}\leq\left\|F(\cdot, t)\right\|_{T^{\infty,2}(\mathbb{R}^{n+1}_{+})}.$$
\end{lem}
 \begin{proof}
We can see that for any $t>0$ ,
 \begin{align*}
& t^{-n/2} \int_0^{t}\int_{B(x, \sqrt{t})}  \left|RF(\cdot, s)\right|^2
{dy ds }\\
 &\ \ \ \ \ \le \left(t^{-n/2} \int_0^{t}\int_{|y-x|<\sqrt{t},y\in\mathbb{R}_{+}^{n}}  \left|\int_{s}^{\infty} e^{(s+\tau)\Delta}\tau^{-\frac{1}{2}}div F(\cdot, \tau)d\tau\right|^2
{dy ds }\right)\\
&\ \ \ \ \ \ \ \ \ \ +\left( t^{-n/2} \int_0^{t}\int_{|y-x|<\sqrt{t},y\in\mathbb{R}_{-}^{n}}  \left|\int_{s}^{\infty} e^{(s+\tau)\Delta}\tau^{-\frac{1}{2}}div F(\cdot, \tau)d\tau\right|^2
{dy ds }\right)\\
&\ \ \ \ \ \le \frac{1}{2}\left\|s^{-1/2}\int_{0}^{2s}e^{\mu \Delta}div f_{+,e}d\mu \right\|_{T^{\infty,2}(\mathbb{R}^{n})}+
\frac{1}{2}\left\|s^{-1/2}\int_{0}^{2s}e^{\mu \Delta}div f_{-,e}d\mu \right\|_{T^{\infty,2}(\mathbb{R}^{n})}\\
&\ \ \ \ \ \le\frac{1}{2}\left\| f_{+,e}\right\|_{T^{\infty,2}(\mathbb{R}^{n+1}_{+})}+\frac{1}{2}
\left\| f_{-,e} \right\|_{T^{\infty,2}(\mathbb{R}^{n+1}_{+})},
\end{align*}
 which completes the proof of Lemma \ref{leee2.e5}.
 \end{proof}

  \smallskip

\begin{prop} \label{leA2.2} Let $u, v\in \varepsilon$. Then we have that \begin{align*}
\left\|\mathcal{A}(\alpha)\right\|_{T^{\infty,2}(\mathbb{R}^{n+1}_{+})}
\leq C\left\|\alpha\right\|_{T^{\infty,1}(\mathbb{R}^{n+1}_{+})}
+C\left\|s^{\frac{1}{2}}\alpha(s)\right\|_{T^{\infty,2}(\mathbb{R}^{n+1}_{+})}.
\end{align*}
\end{prop}
\begin{proof}
For $\mathcal{A}(\alpha)$, it follows from Lemmas \ref{le2.34}-\ref{leee2.e5} that
\begin{align*}
\mathcal{A}(\alpha)(t)
&\leq\mathcal{A}_{1}(\alpha)(t)+\mathcal{A}_{2}(\alpha)(t)
+\mathcal{A}_{3}(\alpha)(t)\\
&\leq \left\| T_{s}\left(s^{1/2}\alpha(s, . )\right)\right\|_{T^{\infty,2}(\mathbb{R}^{n+1}_{+})}+\left\|s^{1/2}\alpha(s, . )\right\|_{T^{\infty,2}(\mathbb{R}^{n+1}_{+})}+
\sup_{\|G\|_{T^{1,2}(\mathbb{R}^{n+1}_{+})}\leq1}\left<\alpha(s, . ), \mathcal{A}_{2}^{*}G\right>\\
&\leq\left\| T_{s}\left(s^{1/2}\alpha(s, . )\right)\right\|_{T^{\infty,2}(\mathbb{R}^{n+1}_{+})}+\left\|s^{1/2}\alpha(s, . )\right\|_{T^{\infty,2}(\mathbb{R}^{n+1}_{+})}\\
&\ \ \ \ \ \ \ \ +\sup_{\|G\|_{T^{1,2}(\mathbb{R}^{n+1}_{+})}\leq1}
\left\|(\mathcal{A}_{2}^{*}G)(s,\cdot)\right
\|_{T^{1,\infty}(\mathbb{R}^{n+1}_{+})}
\left\|\alpha\right\|_{T^{\infty,1}(\mathbb{R}^{n+1}_{+})}\\
&\leq C\left\|\alpha\right\|_{T^{\infty,1}(\mathbb{R}^{n+1}_{+})}
+C\left\|s^{\frac{1}{2}}\alpha(s)\right\|_{T^{\infty,2}(\mathbb{R}^{n+1}_{+})}.
\end{align*}
This completes the proof of Proposition \ref{leA2.2}.
\end{proof}

\begin{proof}[ Proof of  the Theorem  \ref{th1.2}.] Now we define the operator
 $$\Theta(u): =e^{t\Delta_{N}}u_{0}-\int_{0}^{t}e^{(t-s)\Delta_{N}}div_{x} (u^{2}(s, . )) ds.$$
 For the existence of small global solution, we adopt the space
 $$D_{2\epsilon}=\{u\in\varepsilon; \ \ 	\|u\|_{\varepsilon}< 2\epsilon\}.$$ It suffices to show that $\Theta(u)$ is a contraction operator mapping $\varepsilon$ into itself.
Indeed, it is readily seen that
$$\left\|\exp(t\triangle_{N})u_{0}\right\|_{\varepsilon}\lesssim
\left\|u_{0}\right\|_{\rm BMO^{-1}_{\Delta_{N}}(\mathbb{R}^{n})}
\leq\epsilon$$
by Proposition~\ref{le3332.2} and
provided that
$\left\|u_{0}\right\|_{\rm BMO^{-1}_{\Delta_{N}}(\mathbb{R}^{n})}\leq\epsilon$.
Let $u, v\in \varepsilon$. It follows from
Propositions  \ref{le2.33}, \ref{leq2.2} and  \ref{leA2.2} that
\begin{align*}
\left\|\Theta(u)\right\|_{\varepsilon}
&\ \ \ \leq \left\|u_{0}\right\|_{\rm BMO^{-1}_{\Delta_{N}}(\mathbb{R}^{n})}+
C\sum_{i=1}^{i=3}\left\|\mathcal{A}_{i}(u\times v)(t)
\right\|_{T^{\infty,2}(\mathbb{R}^{n+1}_{+})}
+\left\|u\times v\right\|_{T^{\infty,1}(\mathbb{R}^{n+1}_{+})}
+C\left\|s(u\times v)(s)\right\|_{L^{\infty}((0, \infty)\times \mathbb{R}^{n})}\\
&\ \ \ \leq \left\|u_{0}\right\|_{\rm BMO^{-1}_{\Delta_{N}}(\mathbb{R}^{n})}
+C\left\|(u\times v)_{+,e}\right\|_{T^{\infty,1}(\mathbb{R}^{n})}
+C\left\|s^{\frac{1}{2}}(u\times v)_{+,e}(s)\right\|_{T^{\infty,2}(\mathbb{R}^{n+1}_{+})}
+\left\|s(u\times v)_{+,e}(s)\right\|_{L^{\infty}((0, \infty)\times \mathbb{R}^{n})}\\
&\ \ \ \ \ \ +C\left\|(u\times v)_{-,e}\right\|_{T^{\infty,1}(\mathbb{R}^{n+1}_{+})}
+C\left\|s^{\frac{1}{2}}(u\times v)_{-,e}(s)\right\|_{T^{\infty,2}(\mathbb{R}^{n+1}_{+})}
+\left\|s(u\times v)_{-,e}(s)\right\|_{L^{\infty}((0, \infty)\times \mathbb{R}^{n})}\\
&\ \ \ \leq \left\|u_{0}\right\|_{BMO^{-1}_{\Delta_{N}}(\mathbb{R}^{n})}
+C\left\|u\right\|_{T^{\infty,2}(\mathbb{R}^{n+1}_{+})}^{2}
+C\left\|s^{\frac{1}{2}}u\right\|_{L^{\infty}((0, \infty)\times \mathbb{R}^{n})}\left\|u\right\|_{T^{\infty,2}(\mathbb{R}^{n+1}_{+})}
+\left\|s^{\frac{1}{2}}u\right\|_{L^{\infty}((0, \infty)\times \mathbb{R}^{n})}^{2}\\
&\ \ \ \le \epsilon+ C\|u\|_{\varepsilon_{\infty}}^{2}<2\epsilon
\end{align*}
and
\begin{align*}
&\|\Theta(u)-\Theta(v)\|_{\varepsilon}\\
 &\ \ \ \le \|\Phi(u-v,u)\|_{\varepsilon}+\|\Phi(v,u-v)\|_{\varepsilon}\\
&\ \ \ \le  \|(u-v)u\|_{T^{\infty,1}(\mathbb{R}^{n+1}_{+})}+ \|(u-v)v\|_{T^{\infty,1}(\mathbb{R}^{n+1}_{+})}+ \|s(u-v)u\|_{L^{\infty}}+ \|s(u-v)v\|_{L^{\infty}}\\
&\ \ \    \ \ \ + \|s^{1/2}(u-v)v\|_{T^{\infty,2}(\mathbb{R}^{n+1}_{+})}+ \|s^{1/2}(u-v)v\|_{T^{\infty,2}(\mathbb{R}^{n+1}_{+})}\\
&\ \ \ \le  \|u-v\|_{T^{\infty,2}}\|(u,v)\|_{T^{\infty,2}(\mathbb{R}^{n+1}_{+})}+ \|s^{\frac{1}{2}}(u-v)\|_{L^{\infty}}(\|s^{\frac{1}{2}}v\|_{L^{\infty}}+\|s^{\frac{1}{2}}u\|_{L^{\infty}})\\
&\ \ \    \ \ \ + \|s^{1/2}(u-v)\|_{L^{\infty}}^{\frac{1}{2}}
\|s^{1/2}v\|_{L^{\infty}}^{\frac{1}{2}}\|u-v\|_{T^{\infty,2}(\mathbb{R}^{n+1}_{+})}^{\frac{1}{2}}
\|v\|_{T^{\infty,2}(\mathbb{R}^{n+1}_{+})}^{\frac{1}{2}}\\
&\ \ \    \ \ \ + \|s^{1/2}(u-v)\|_{L^{\infty}}^{\frac{1}{2}}
\|s^{1/2}u\|_{L^{\infty}}^{\frac{1}{2}}\|u-v\|_{T^{\infty,2}(\mathbb{R}^{n+1}_{+})}^{\frac{1}{2}}
\|u\|_{T^{\infty,2}(\mathbb{R}^{n+1}_{+})}^{\frac{1}{2}}\\
&\ \ \ \le  C\|u-v\|_{\varepsilon}\|(u,v)\|_{\varepsilon}\leq
 C\epsilon\|u-v\|_{\varepsilon},
\end{align*}
provied that $u\in D_{2\epsilon}$ and $C\epsilon<1$, the map $\Theta: D_{2\epsilon}\rightarrow D_{2\epsilon}$ is a contraction and has a fixed point in
$D_{2\epsilon}$, which is the unique solution $u$ for the integral equation satisfying $\|u\|_{ \varepsilon}\leq 2\epsilon$. Hence there exists a unique solution $u \in \varepsilon$ satisfying $u =\Theta(u)$ due to the Banach contraction principle. The proof of Theorem  \ref{th1.2} is complete.
\end{proof}
 \medskip



\begin{thebibliography}{10}

\bibitem{AF1} P. Auscher and D. Frey,  On the Well-Posedness of parabolic equations of  Navier-Stokes type with $BMO^{-1}$ data.
       {\it J. Inst. Math. Jussieu.,} {\bf 16(5)} (2017), 947-985.

 \bibitem {DKP} M. Dindos, C. Kenig and J.  Pipher,
  BMO solvability and the $A_{\infty}$ condition for elliptic operators. \textit{J. Geom. Anal. } \textbf{21} (2011), 78--95.


\bibitem{9} P. Auscher, S. Monniaux and P. Portal, The maximal regularity operator on tent spaces. {\it Commun. Pure Appl. Anal.} {\bf 11(6)} (2012), 2213-2219.


\bibitem{12} P. Auscher and P. Tchamitchian. {\it Square root problem for divergence operators and related topics. Ast$\acute{e}$risqueno.} 249, 1998.

\bibitem{DDSY}
D. Deng, X. Duong, A. Sikora and L. Yan, Comparison of the classical BMO with the BMO spaces associated with operators and applications. {\it Rev. Mat. Iberoamericana.} {\bf 24(1)} (2008), 267-296.

\bibitem {DYZ}X. Duong, L. Yan and C. Zhang,
On characterization of Poisson integrals of Schr\"odinger operators with BMO traces. {\it J. Funct. Anal.}{\bf 266}(2014),  2053-2085.








\bibitem{FJN}  E. Fabes, R. Johnson and U. Neri,
Spaces of harmonic functions representable by Poisson integrals of functions in BMO and $L_{p, \lambda}$.
\textit{ Indiana Univ. Math. J.}
\textbf{ 25} (1976), 159--170.


\bibitem{FN1} E. Fabes and U. Neri,
 Characterization of temperatures with initial data in BMO. \textit{ Duke Math. J. }
\textbf{42} (1975),  725-734.

\bibitem{FN} E. Fabes and U. Neri,   Dirichlet problem in Lipschitz domains with BMO data.
\textit{Proc. Amer. Math. Soc.}
\textbf{78} (1980), 33--39.



\bibitem{FS} C. Fefferman and E.M. Stein,   $H^p$ spaces of
 several variables. {\it Acta
Math.} {\bf 129} (1972), 137--195.




%

\bibitem{JX} R. Jiang, J. Xiao and D. Yang, Towards spaces of harmonic functions with traces in square Campanato spaces and their scaling invariants, \textit{Anal. Appl. (Singap.)} \textbf{14} (2016), 679--703.


\bibitem {LW}
J. Li and B. Wick,
Characterizations of $H^{1}_{\Delta_{N}}(\mathbb{R}^{n})$ and $BMO_{\Delta_{N}}(\mathbb{R}^{n})$
via weak factorizations and commutators, {\it J. Funct. Anal.}{\bf 272}(2017),  5384-5416.
















\bibitem{Song} L. Song, X. Tian and L. Yan, On characterization of Poisson integrals of Schr\"odinger operators with Morry traces, {\it arXiv:1506.05239v3}.

\bibitem{St1970} E. M. Stein,
\textit{Topics in Harmonic Analysis Related to the Littlewood-Paley Theory},
{Annals of Mathematics Studies} \textbf{63},
Princeton Univ. Press,
Princeton, NJ, 1970.




\bibitem{SW} E. M. Stein and G. Weiss,
\textit{Introduction to Fourier Analysis on Euclidean spaces},
Princeton Univ. Press,
Princeton, NJ, 1970.




 \bibitem{YSY} W. Yuan, W. Sickel and D. Yang,
\textit{Morrey and Campanato meet Besov, Lizorkin and Triebel},
Lecture Notes in Mathematics, 2005. Springer-Verlag, Berlin, 2010.






\end{thebibliography}
\end{document}